\documentclass[11pt]{amsart}
\usepackage{amssymb}
\usepackage[all]{xy}
\usepackage[toc,page,title,titletoc,header]{appendix}
\usepackage{amsmath,amssymb,epsfig,amscd,amsthm,verbatim}
\usepackage[bookmarks,colorlinks,breaklinks]{hyperref}
\usepackage{verbatim}

\begin{document}
\theoremstyle{plain}
\newtheorem{thm}{Theorem}[section]
\newtheorem*{thm1}{Theorem 1}
\newtheorem*{thm1.1}{Theorem 1.1}
\newtheorem*{thmM}{Main Theorem}
\newtheorem*{thmA}{Theorem A}
\newtheorem*{thm2}{Theorem 2}
\newtheorem{lemma}[thm]{Lemma}
\newtheorem{lem}[thm]{Lemma}
\newtheorem{cor}[thm]{Corollary}
\newtheorem{pro}[thm]{Proposition}
\newtheorem{propose}[thm]{Proposition}
\newtheorem{variant}[thm]{Variant}
\theoremstyle{definition}
\newtheorem{notations}[thm]{Notations}
\newtheorem{rem}[thm]{Remark}
\newtheorem{rmk}[thm]{Remark}
\newtheorem{rmks}[thm]{Remarks}
\newtheorem{defi}[thm]{Definition}
\newtheorem{exe}[thm]{Example}
\newtheorem{claim}[thm]{Claim}
\newtheorem{ass}[thm]{Assumption}
\newtheorem{prodefi}[thm]{Proposition-Definition}
\newtheorem{que}[thm]{Question}
\newtheorem{con}[thm]{Conjecture}

\newtheorem*{smlprob}{Problem Skolem-Mahler-Lech}
\newtheorem*{pvprob}{Problem Picard-Vessiot}
\newtheorem*{tvcon}{Tate-Voloch Conjecture}
\newtheorem*{dmmcon}{Dynamical Manin-Mumford Conjecture}
\newtheorem*{dmlcon}{Dynamical Mordell-Lang Conjecture}
\newtheorem*{condml}{Dynamical Mordell-Lang Conjecture}
\numberwithin{equation}{section}
\newcounter{elno}                
\def\points{\list
{\hss\llap{\upshape{(\roman{elno})}}}{\usecounter{elno}}}
\let\endpoints=\endlist

\newcommand{\lra}{{\longrightarrow}}
\newcommand{\dra}{{\dashrightarrow}}
\newcommand{\Gr}{{\rm Gr}}
\newcommand{\GO}{{\rm GO}}
\newcommand{\Fan}{{(\F^{\an})}}
\newcommand{\Phian}{{(\Phi^{\an})}}
\newcommand{\lcm}{{\rm lcm}}
\newcommand{\Tor}{{\rm Tor}}
\newcommand{\perf}{{\rm perf}}
\newcommand{\ad}{{\rm ad}}
\newcommand{\Spa}{{\rm Spa}}
\newcommand{\Perf}{{\rm Perf}}
\newcommand{\alHom}{{\rm alHom}}
\newcommand{\SH}{\rm SH}
\newcommand{\Tan}{\rm Tan}
\newcommand{\res}{{\rm res}}
\newcommand{\Om}{\Omega}
\newcommand{\om}{\omega}
\newcommand{\OO}{\mathcal{O}}
\newcommand{\cO}{\mathfrak{O}}
\newcommand{\cm}{\mathfrak{m}}
\newcommand{\la}{\lambda}
\newcommand{\mc}{\mathcal}
\newcommand{\mb}{\mathbb}
\newcommand{\surj}{\twoheadrightarrow}
\newcommand{\inj}{\hookrightarrow}
\newcommand{\zar}{{\rm zar}}
\newcommand{\Exc}{\rm Exc}
\newcommand{\an}{{\rm an}}
\newcommand{\red}{{\rm \mathbf{red}}}
\newcommand{\codim}{{\rm codim}}
\newcommand{\Supp}{{\rm Supp}}
\newcommand{\rank}{{\rm rank}}
\newcommand{\Ker}{{\rm Ker \ }}
\newcommand{\Pic}{{\rm Pic}}
\newcommand{\Div}{{\rm Div}}
\newcommand{\Hom}{{\rm Hom}}
\newcommand{\im}{{\rm im}}
\newcommand{\Spec}{{\rm Spec \,}}
\newcommand{\Nef}{{\rm Nef \,}}
\newcommand{\Frac}{{\rm Frac \,}}
\newcommand{\Sing}{{\rm Sing}}
\newcommand{\sing}{{\rm sing}}
\newcommand{\reg}{{\rm reg}}
\newcommand{\Char}{{\rm char}}
\newcommand{\Fix}{{\rm Fix}}
\newcommand{\Tr}{{\rm Tr}}
\newcommand{\ord}{{\rm ord}}
\newcommand{\id}{{\rm id}}
\newcommand{\NE}{{\rm NE}}
\newcommand{\Gal}{{\rm Gal}}
\newcommand{\Min}{{\rm Min \ }}
\newcommand{\Max}{{\rm Max \ }}
\newcommand{\Alb}{{\rm Alb}\,}
\newcommand{\GL}{{\rm GL}\,}        
\newcommand{\PGL}{{\rm PGL}\,}
\newcommand{\Bir}{{\rm Bir}}
\newcommand{\Aut}{{\rm Aut}}
\newcommand{\fp}{\mathfrak{p}}
\newcommand{\End}{{\rm End}}
\newcommand{\Per}{{\rm Per}\,}
\newcommand{\ie}{{\it i.e.\/},\ }
\newcommand{\niso}{\not\cong}
\newcommand{\nin}{\not\in}
\newcommand{\soplus}[1]{\stackrel{#1}{\oplus}}
\newcommand{\by}[1]{\stackrel{#1}{\rightarrow}}
\newcommand{\longby}[1]{\stackrel{#1}{\longrightarrow}}
\newcommand{\vlongby}[1]{\stackrel{#1}{\mbox{\large{$\longrightarrow$}}}}
\newcommand{\ldownarrow}{\mbox{\Large{\Large{$\downarrow$}}}}
\newcommand{\lsearrow}{\mbox{\Large{$\searrow$}}}
\renewcommand{\d}{\stackrel{\mbox{\scriptsize{$\bullet$}}}{}}
\newcommand{\dlog}{{\rm dlog}\,}    
\newcommand{\longto}{\longrightarrow}
\newcommand{\vlongto}{\mbox{{\Large{$\longto$}}}}
\newcommand{\limdir}[1]{{\displaystyle{\mathop{\rm lim}_{\buildrel\longrightarrow\over{#1}}}}\,}
\newcommand{\liminv}[1]{{\displaystyle{\mathop{\rm lim}_{\buildrel\longleftarrow\over{#1}}}}\,}
\newcommand{\norm}[1]{\mbox{$\parallel{#1}\parallel$}}
\newcommand{\boxtensor}{{\Box\kern-9.03pt\raise1.42pt\hbox{$\times$}}}
\newcommand{\into}{\hookrightarrow}
\newcommand{\image}{{\rm image}\,}
\newcommand{\Lie}{{\rm Lie}\,}      
\newcommand{\CM}{\rm CM}
\newcommand{\sext}{\mbox{${\mathcal E}xt\,$}}  
\newcommand{\shom}{\mbox{${\mathcal H}om\,$}}  
\newcommand{\coker}{{\rm coker}\,}  
\newcommand{\sm}{{\rm sm}}
\newcommand{\pgcd}{\text{pgcd}}
\newcommand{\trd}{\text{tr.d.}}
\newcommand{\tensor}{\otimes}
\renewcommand{\iff}{\mbox{ $\Longleftrightarrow$ }}
\newcommand{\supp}{{\rm supp}\,}
\newcommand{\ext}[1]{\stackrel{#1}{\wedge}}
\newcommand{\onto}{\mbox{$\,\>>>\hspace{-.5cm}\to\hspace{.15cm}$}}
\newcommand{\propsubset}
{\mbox{$\textstyle{
\subseteq_{\kern-5pt\raise-1pt\hbox{\mbox{\tiny{$/$}}}}}$}}
\newcommand{\sA}{{\mathcal A}}
\newcommand{\sB}{{\mathcal B}}
\newcommand{\sC}{{\mathcal C}}
\newcommand{\sD}{{\mathcal D}}
\newcommand{\sE}{{\mathcal E}}
\newcommand{\sF}{{\mathcal F}}
\newcommand{\sG}{{\mathcal G}}
\newcommand{\sH}{{\mathcal H}}
\newcommand{\sI}{{\mathcal I}}
\newcommand{\sJ}{{\mathcal J}}
\newcommand{\sK}{{\mathcal K}}
\newcommand{\sL}{{\mathcal L}}
\newcommand{\sM}{{\mathcal M}}
\newcommand{\sN}{{\mathcal N}}
\newcommand{\sO}{{\mathcal O}}
\newcommand{\sP}{{\mathcal P}}
\newcommand{\sQ}{{\mathcal Q}}
\newcommand{\sR}{{\mathcal R}}
\newcommand{\sS}{{\mathcal S}}
\newcommand{\sT}{{\mathcal T}}
\newcommand{\sU}{{\mathcal U}}
\newcommand{\sV}{{\mathcal V}}
\newcommand{\sW}{{\mathcal W}}
\newcommand{\sX}{{\mathcal X}}
\newcommand{\sY}{{\mathcal Y}}
\newcommand{\sZ}{{\mathcal Z}}
\newcommand{\A}{{\mathbb A}}
\newcommand{\B}{{\mathbb B}}
\newcommand{\C}{{\mathbb C}}
\newcommand{\D}{{\mathbb D}}
\newcommand{\E}{{\mathbb E}}
\newcommand{\F}{{\mathbb F}}
\newcommand{\G}{{\mathbb G}}
\newcommand{\HH}{{\mathbb H}}
\newcommand{\I}{{\mathbb I}}
\newcommand{\J}{{\mathbb J}}
\newcommand{\M}{{\mathbb M}}
\newcommand{\N}{{\mathbb N}}
\renewcommand{\P}{{\mathbb P}}
\newcommand{\Q}{{\mathbb Q}}
\newcommand{\R}{{\mathbb R}}
\newcommand{\T}{{\mathbb T}}
\newcommand{\U}{{\mathbb U}}
\newcommand{\V}{{\mathbb V}}
\newcommand{\W}{{\mathbb W}}
\newcommand{\X}{{\mathbb X}}
\newcommand{\Y}{{\mathbb Y}}
\newcommand{\Z}{{\mathbb Z}}

\newcommand{\fix}{\mathrm{Fix}}

\newcommand{\nn}{\mathbb{N}}
\newcommand{\seq}{\operatorname{Seq}}
\newcommand{\s}{\sigma}
\newcommand{\Gl}{\operatorname{GL}}

\title{The Dynamical Mordell-Lang Conjecture for skew-linear self-maps} 

\author{Dragos Ghioca and Junyi Xie, with an appendix by Michael Wibmer}

\address{
Dragos Ghioca\\
Department of Mathematics\\
University of British Columbia\\
Vancouver, BC V6T 1Z2\\
Canada
}
\email{dghioca@math.ubc.ca}

\address{Junyi Xie, IRMAR, Campus de Beaulieu,
b\^atiments 22-23
263 avenue du G\'en\'eral Leclerc, CS 74205
35042  RENNES C\'edex, France}
\email{junyi.xie@univ-rennes1.fr}

\address{Michael Wibmer\\
university of Pennsylvania\\
Department of Mathematics\\
David Rittenhouse Lab., 209 South 33rd Street\\
Philadelphia, PA 19104-6395\\
USA}
\email{wibmer@math.upenn.edu}


\begin{abstract}
Let $k$ be an algebraically closed field of characteristic $0$, let $N\in\N$, let $g:\P^1\lra \P^1$ be a non-constant morphism, and let $A:\A^N\lra \A^N$ be a linear transformation defined over $k(\P^1)$, i.e., for a Zariski open dense subset $U\subset \P^1$, we have that for $x\in U(k)$, the specialization $A(x)$ is an $N$-by-$N$ matrix with entries in $k$. We let $f:\P^1\times \A^N\dra \P^1\times \A^N$ be the rational endomorphism given by $(x,y)\mapsto (g(x), A(x)y)$. We prove that if $g$ induces an automorphism of $\A^1\subset \P^1$, then each irreducible curve $C\subset \A^1\times \A^N$ which intersects some orbit $\OO_f(z)$ in infinitely many points must be periodic under the action of $f$. Furthermore, in the case $g:\P^1\lra \P^1$ is an endomorphism of degree greater than $1$, then we prove that each irreducible subvariety $Y\subset \P^1\times \A^N$ intersecting an orbit $\OO_f(z)$ in a Zariski dense set of points must be periodic. Our results provide the desired conclusion in the Dynamical Mordell-Lang Conjecture in a couple new instances. Also, our results have interesting consequences towards a conjecture of Rubel and towards a generalized Skolem-Mahler-Lech problem proposed by Wibmer in the context of difference equations. In the appendix it is shown that the results can also be used to construct Picard-Vessiot extensions in the ring of sequences. 
\end{abstract}

\thanks{The first author is partially supported by a Discovery Grant from the Natural Sciences and Engineering Research Council of Canada, while the second author is partially supported by project "Fatou" ANR-17-CE40-0002-01.}

\maketitle


\section{Introduction}


\subsection{Notation}

We let $\N_0:=\N\cup\{0\}$. In our paper, we allow an arithmetic progression to have common difference equal to $0$, in which case, the arithmetic progression is just a singleton.

Throughout our paper, we let $k$ be an algebraically closed field of characteristic $0$. Also, unless otherwise noted, all our subvarieties are assumed to be closed. In general, for a set $S$ contained in an algebraic variety $X$, we denote by $\overline{S}$ its Zariski closure.

For a variety $X$ defined over $k$ and endowed with a rational self-map $\Phi$, for any subvariety $V\subseteq X$, we define $\Phi(V)$ to be the Zariski closure of the set $\Phi\left(V\setminus I(\Phi)\right)$, where $I(\Phi)$ is the indeterminacy locus of $\Phi$; in other words, $\Phi(V)$ is the strict transform of $V$ under $\Phi$. Also, we denote by $\OO_\Phi(\alpha)$ the orbit of any point $\alpha\in X(K)$ under $\Phi$, i.e., the set of all $\Phi^n(\alpha)$ for $n\in \N_0$ (as always in algebraic dynamics, we denote by $\Phi^n$ the $n$-th compositional power of the map $\Phi$, where $\Phi^0$ is the identity map, by convention). We say that $\alpha$ is periodic if there exists $n\in\N$ such that $\Phi^n(\alpha)=\alpha$; we say that $\alpha$ is preperiodic if there exists $m\in\N_0$ such that $\Phi^m(\alpha)$ is periodic. More generally, for an irreducible subvariety $V\subset X$, we say that $V$ is periodic if $\Phi^n(V)= V$ for some $n\in\N$; if $\Phi(V)=V$ (i.e., $\overline{\Phi\left(V\setminus I(\Phi)\right)}=V$), we say that $V$ is invariant under the action of $\Phi$ (or simpler, invariant by $\Phi$).

We will also encounter the following setup in our paper. Given a variety $X$ (which we will call the \emph{base}) defined over $k$ and given $N\in\N$, we consider some $N$-by-$N$ matrix $A$ whose entries are rational functions on $X$; when the determinant of $A$ is  nonzero, then we write $A\in \GL_N(k(X))$. For each $N$-by-$N$ matrix $A\in M_{N,N}(k(X))$ there exists an open, Zariski dense subset $U\subset X$ such that for each $x\in U$, the matrix $A(x)$ obtained by evaluating each entry of $A$ at $x$ is well-defined. We call \emph{skew-linear self-map (over the base $X$)} a rational self-map $f:X\times \mathbb{A}^N\dashrightarrow X\times \mathbb{A}^N$ of the form $f(x, y)=(g(x), A(x)y)$, where $g:X\dashrightarrow X$ is a given rational self-map, while $A\in M_{N,N}(k(X))$. For more details regarding skew-linear self-maps and their application to arithmetic dynamics, we refer the reader to \cite{skew linear paper 1}. In our paper we will work with skew-linear self-maps over the base $X=\A^1$ or $X=\P^1$.


\subsection{The Dynamical Mordell Lang Conjecture}


Motivated by the famous Mordell-Lang conjecture (now a theorem, due to Faltings \cite{Faltings}), the Dynamical Mordell-Lang Conjecture (see \cite{Ghioca2009}) predicts the following: given a quasiprojective variety $X$ defined over $k$, endowed with an endomorphism $\Phi$, and also given a point $\alpha\in X(k)$ and a subvariety $Y\subset X$, the set $\left\{n\in\N_0\colon \Phi^n(\alpha)\in Y(k)\right\}$ 
is a finite union of arithmetic progressions. In the special case when $Y$ is  irreducible and of positive dimension (note that the conjecture is trivial when $\dim(Y)=0$), the Dynamical Mordell-Lang Conjecture is equivalent with asking that if $Y\cap\OO_\Phi(\alpha)$ is Zariski dense in $Y$, then $Y$ must be periodic (under the action of $\Phi$); for more details, see \cite[Chapter~3]{Bell2016}.

There are several partial results supporting this conjecture and no counterexamples known (see \cite{Bell2016} for a survey of known results prior to 2016). Almost all known results towards the Dynamical Mordell-Lang Conjecture (with only a few outstanding exceptions, such as the results from \cite{Bell2010} for all \'etale endomorphisms of any quasiprojective variety and the results from \cite{Xiec} for all endomorphisms of $\A^2$, which in turn, extend the results from \cite{Xie-Math}) are valid for \emph{split endomorphisms} only; more precisely, $X=(\P^1)^N$ and the endomorphism $\Phi$ is given by the coordinatewise action of $N$ rational functions $\varphi_i$, i.e.
$$(x_1,\dots, x_N)\mapsto \left(\varphi_1(x_1),\dots, \varphi_N(x_N)\right).$$


\subsection{Our results}


In this paper, we study the Dynamical Mordell-Lang Conjecture for skew-linear self-maps whose base is a rational curve. When the action on the base is linear, we have the following result (proven in Section~\ref{section proofs 1}).

\begin{thm}\label{thmdmlslao}
Let $f:\A^1_k\times\A^N_k\to \A^1_k\times\A^N_k$ be an endomorphism defined by $(x,y)\mapsto (g(x),A(x)y)$ where $g$ is an automorphism of $\A^1_k$, while  $A(x)$ is a matrix in $M_{N\times N}(k[x])$. Let $C$ be an irreducible curve in $\A^1_k\times\A^N_k$ and $\alpha$ be a point in $\A^1_k\times\A^N_k$. If  $\OO_f(\alpha)\cap C$ is infinite, then $C$ is periodic under $f$.
\end{thm}

When the action on the base is non-linear, we get a stronger result (proven in Section~\ref{section proofs 2}).
\begin{thm}\label{thmdmlsenonlinear}
Let $f:\P^1_k\times \A^N_k\dashrightarrow \P^1_k\times \A^N_k$ be a rational self-map of the form $(x,y)\mapsto (g(x),A(x)(y))$ where $g$ is an endomorphism of $\P^1_k$ of degree strictly greater then one, while  $A(x)$ is a matrix in $M_{N\times N}(k(x))$. Let $Y$ be an irreducible subvariety in $\P^1_k\times \A^N_k$ of positive dimension and let $\alpha$ be a point in $\P^1_k\times \A^N_k$. If $f^n(\alpha)\not\in I(f)$ for all $n\geq 0$ and  $\OO_f(\alpha)\cap Y$ is Zariski dense in $Y$, then $Y$ is periodic under $f$.
\end{thm}

\medskip

By considering the isomorphism $\P^N_k=\A^{N+1}_k/\G_{m,k}$,  
Theorem \ref{thmdmlsenonlinear} implies the following result immediately.  
\begin{cor}\label{cordmlprolin}
Let $f:\P^1_k\times \P_k^N\dashrightarrow \P^1_k\times \P^N_k$  be a dominant rational self-map of the form $(x,y)\mapsto (g(x),A(x)(y))$, where $g$ is an endomorphism of $\P^1_k$ of degree strictly greater then one, while  $A(x)$ is an element in $\PGL_{N+1}(k(x))$. Let $Y$ be an irreducible subvariety in $\P^1_k\times \P^N_k$ of positive dimension and let $\alpha$ be a point in $\P^1_k\times \P^N_k$. If $f^n(\alpha)\not\in I(f)$ for all $n\geq 0$ and  $\OO_f(\alpha)\cap Y$ is Zariski dense in $Y$, then $Y$ is periodic under $f$.
\end{cor}

Our Theorems~\ref{thmdmlslao}~and~\ref{thmdmlsenonlinear} are some of the very few known instances (besides the case of \'etale endomorphisms proven in \cite{Bell2010} and the case  of endomorphisms of $\A^2$ proven in \cite{Xiec}) when the Dynamical Mordell-Lang Conjecture is proven for a non-split endomorphism.

Our proofs for Theorems~\ref{thmdmlslao} and \ref{thmdmlsenonlinear} are quite different. While for Theorem~\ref{thmdmlsenonlinear} we are able to show that there exists a suitable prime with respect to which one can use the \emph{$p$-adic arc lemma} (for more details on this construction, see \cite[Chapter~4]{Bell2016}), in order to prove Theorem~\ref{thmdmlslao}, our main tool is Siegel's theorem along with the classical Quillen-Suslin theorem (see \cite{Quillen}). On the other hand, we explain in Remark~\ref{rem:no extension} that our proof for Theorem~\ref{thmdmlsenonlinear} cannot be extended to the case $g$ is an automorphism of the base in order to treat intersections of orbits with subvarieties $Y\subset \A^1\times \A^N$ of dimension larger than $1$.  Finally, we note that if the rational function $g$ from Theorems~\ref{thmdmlslao}~and~\ref{thmdmlsenonlinear} were constant, then the conclusion of the Dynamical Mordell-Lang Conjecture would follow immediately in this special case, as a consequence of \cite{Bell2010}.  


\subsection{Further applications}


Theorem \ref{thmdmlsenonlinear} has the following applications to linear difference equations.

Let $g\in k(x)$ be any nonconstant rational function and let $\ell\in \N$. Then $g$ defines a difference field $(k(x), \sigma)$ where $\sigma$ is the endomomorphism of $k(x)$ defined by $\sigma(h(x))=h(g(x))$ for $h\in k(x)$.
In \cite{Wib-JEMS}, Wibmer studied the following two problems in the case that $g(x)=x+1.$

\begin{smlprob} 
Let $\{a_n\}_{n\geq 0}$ be a 
recurrence sequence in $k$ satisfying the following recurrence equation 
$$a_{n+\ell}=\sum_{i=0}^{\ell-1} h_i(g^{n}(\alpha))a_{n+i}$$ 
where the $h_i$'s are rational functions in $k(x)$, while $\alpha\in\P_k^1$. 
Is it true that the set $\{n\geq 0\colon a_n=0\}$ is a finite union of arithmetic progressions?  
\end{smlprob}

\begin{pvprob}
Does there exist a Picard-Vessiot extension of $k(x)$ for the linear difference equation $\sigma^\ell(y)-h_{\ell-1}\sigma^{\ell-1}(y)-\dots-h_{0}y=0$ inside the
ring of $k$-valued sequences?
\end{pvprob}

For more background on difference equations and the aforementioned two problems, we refer the reader to \cite{Wib-JEMS} and \cite{Put-Singer}.

In \cite{Wib-JEMS}, Wibmer shows that when $g(x)=x+1$, a certain special case of the Dynamical Mordell-Lang Conjecture would imply an affirmative solution to problems Skolem-Mahler-Lech and Picard-Vessiot and he also solved Problem Picard-Vessiot affirmatively under the restriction $h_{1},\dots,h_{\ell-1}\in k[x]$ and $h_0\in k\setminus \{0\}.$

As a direct application of Theorem~\ref{thmdmlsenonlinear} 
 we solve Problem Skolem-Mahler-Lech affirmatively when $\deg g\geq 2$.
\begin{cor}\label{corsmlnonline}
Let $g\in k(x)$ be a rational function of degree at least $2$, let $\alpha\in \P_k^1$ and let $\ell\in \N$. Let $\{a_n\}_{n\geq 0}$ be a recurrence sequence in $k$ satisfying the following recurrence equation 
\begin{equation}
\label{Rubel equation}
a_{n+\ell}=\sum_{i=0}^{\ell-1} h_i(g^{n}(\alpha))a_{n+i} 
\end{equation}
where $h_i$ are rational functions in $k(x).$
Then the set $\{n\geq 0\colon a_n=0\}$ is a finite union of arithmetic progressions.
\end{cor}
Indeed, Corollary~\ref{corsmlnonline} follows as a consequence of our Theorem~\ref{thmdmlsenonlinear} in a similar way as \cite[Theorem~1.7]{BGT-Func} followed from \cite[Corollary~1.5]{BGT-Func}. More precisely, we consider the rational map $\Phi:\P^1\times \A^\ell\dashrightarrow \P^1\times\A^\ell$ defined by $\Phi(x,y)=(g(x),A(x)y)$, where the linear transformation $A\in M_{\ell,\ell}(k(x))$ is given by
\[A(x):=\left(\begin{array}{cccccc}
0 & 1 & 0 & 0 & \cdots & 0\\
0 & 0 & 1 & 0 & \cdots & 0 \\
\cdots & \cdots & \cdots & \cdots & \cdots & \cdots \\
0 & 0 & 0 & \cdots & 0 & 1 \\ 
h_0(x) & h_1(x) & \cdots & \cdots & \cdots & h_{\ell-1}(x)
\end{array}\right) 
\]
Then letting $v_0:=(a_0,a_1,\cdots,a_{\ell-1})$, we have that  
$$\Phi^n(\alpha,v_0) = \left(g^n(\alpha),a_n,a_{n+1},\cdots,a_{n+\ell-1}\right).$$ 
So, letting $Y_1:=\{0\}\times \A^{\ell-1}\subset \A^\ell$ and then letting $Y:=\P^1\times Y_1$, allows us to apply Theorem~\ref{thmdmlsenonlinear} to the subvariety $Y\subset \P^1\times \A^\ell$ under the action of $\Phi$ in order to  derive the desired conclusion in Corollary~\ref{corsmlnonline}.

We also observe that our Corollary~\ref{corsmlnonline} yields a positive answer to a variant of Rubel's \cite[Question~16]{Rubel} (see also \cite[Theorem~1.7,~p.~3-4]{BGT-Func}). Indeed, \cite[Question~16]{Rubel} asks to characterize the set of all $n\in\N_0$ such that $a_n=0$, where 
$$f(z):=\sum_{n=0}^\infty a_nz^n$$
is the solution of a linear differential equation with polynomial coefficients. For the question raised by Rubel \cite{Rubel}, the sequence $\{a_n\}$ satisfies a recurrence sequence of the form \eqref{Rubel equation} where $g(x)=x+1$. 

\medskip

The role of Picard-Vessiot extensions in the Galois theory of linear difference equations is similar to the role of splitting fields in the usual Galois theory of polynomials. Instead of requiring that a polynomial of degree $\ell$ has $\ell$ distinct roots in the splitting field one requires that the $k$-space of all solutions to  $\sigma^\ell(y)-h_{\ell-1}\sigma^{\ell-1}(y)-\dots-h_{0}y=0$ in the Picard-Vessiot extension has dimension $\ell$. The recurrence formula (\ref{Rubel equation}) yields $\ell$ $k$-linearly independent solutions in the ring of $k$-valued sequences. It is therefore natural to ask if there exists a Picard-Vessiot extension inside the ring of sequences. In the appendix (Section \ref{sec:appendix}) we will show how Corollary \ref{corsmlnonline} can be used to solve problem Picard-Vessiot affirmative if $\deg g\geq 2$.

%

\begin{thm}\label{thmpvnonline} If $\deg g\geq 2$ and $h_{\ell-1},\ldots,h_0\in k(x)$, then there exists a Picard-Vessiot extension of $k(x)$ for $\sigma^\ell(y)-h_{\ell-1}\sigma^{\ell-1}(y)-\dots-h_{0}y=0$ inside the ring of $k$-valued sequences.
\end{thm}
Theorem \ref{thmpvnonline} follows from Theorem \ref{theo:existsPV} by choosing $A$ as in the proof of Corollary~\ref{corsmlnonline}.

%
%

\section{Proof of Theorem \ref{thmdmlslao}}
\label{section proofs 1}


We work under the hypothesis of Theorem~\ref{thmdmlslao}. We prove our result by induction on $N$, noting that the case $N=0$ is trivial. 

We also note that in our proof we may replace $f$ by an iterate of itself. In addition, we may always replace $g$ by a conjugate of itself through a linear automorphism, and therefore, replace $f$ by a conjugate through an automorphism of $\A^1\times\A^N$; for more details regarding the various reductions in the Dynamical Mordell-Lang Conjecture, see also \cite[Chapter~3]{Bell2016}.

We proceed by first dealing with the case when $\det A(x)$ is constant. 

\begin{lemma}
\label{determinant constant automorphism}
Theorem~\ref{thmdmlslao} holds if $\det A(x)$ is constant. 
\end{lemma}

\begin{proof}[Proof of Lemma~\ref{determinant constant automorphism}.]
If $\det A(x)$ is a nonzero constant, then $f$ is an automorphism and so, the result follows from \cite{Bell2010}. 

So, assume now that $\det A(x)=0$. We let $\ell$ be the rank of $A(x)$ as a matrix of elements in $k(x).$ We have $\ell\leq N-1.$ Let $Y$ be the Zariski closure of the image of $f$. Then $Y$ is a subbundle of $\A^1\times\A_k^N$ of rank $\ell$ and it is invariant under $f$.  Since $f(\alpha)\in Y$, after replacing $\alpha$ by $f(\alpha)$, we may suppose that $\alpha$ is contained in $Y$. Then $f^n(\alpha)\in Y$ for all $n\geq 0.$
If $C$ is not contained in $Y$, then $\OO_f(\alpha)\cap C \subseteq C\cap Y$ is finite.  Otherwise, we have $C\subseteq Y$.
Quillen-Suslin's theorem (see \cite{Quillen}) yields that all vector bundles on $\A^1_k$ are trivial. So, $Y$ is isomorphic to $\A^1_k\times \A^{N-1}_k$, which allows us to conclude our proof of Theorem~\ref{thmdmlslao} by the induction hypothesis.
\end{proof}

{\bf Therefore, from now on, we assume that $\det A$ is a non-constant polynomial in $k[x]$.} 

\begin{lemma}
\label{N=1 automorphism}
Theorem~\ref{thmdmlslao} holds when $N=1$.
\end{lemma}

\begin{proof}[Proof of Lemma~\ref{N=1 automorphism}.]
Since $\det A(x)\ne 0$, then $f$ is a birational automorphism of $\A^2$; hence,  this case is covered by the result of \cite{Xie-Math}.
\end{proof}

{\bf Therefore, from now on, we may assume that $N\ge 2$.}

\begin{lemma}
\label{a root of unity automorphism}
Theorem~\ref{thmdmlslao} holds if $g(x)=x$.
\end{lemma}

\begin{proof}[Proof of Lemma~\ref{a root of unity automorphism}.]
Theorem \ref{thmdmlslao} holds in this case, since it reduces to proving the Dynamical Mordell-Lang Conjecture for a linear self-map on $\A^N$. Indeed, given a linear map $\psi:\A^N\lra \A^N$ and a point $\beta\in \A^N(k)$,  at the expense of replacing $\psi$ by an iterate (and also replacing $\beta$ by a suitable $\psi^m(\beta)$) we can find a subvariety $\beta \in W\subset \A^N$ invariant under $\psi$ and moreover, $\psi|_W:W\lra W$ is an automorphism; then \cite[Theorem~4.1]{Bell2010} provides the desired conclusion.  
\end{proof}

{\bf From now on, we let $g(x)=ax+b$ with $a\in k^*$ and $b\in k$.}

If $a$ is a root of unity, then at the expense of replacing $f$ by an iterate, we may assume $a=1$. Since the case $g(x)=x$ was proven in Lemma~\ref{a root of unity automorphism}, we are left with the possibility that $g(x)=x+1$ (again, at the expense of replacing $g$ and therefore $f$ by a conjugate through a suitable automorphism).

If $a$ is not a root of unity, then after a suitable conjugation by an automorphism, we may assume $g(x)=ax$. 

{\bf Therefore, from now on, we may assume that either $g(x)=x+1$, or $g(x)=ax$ for some $a\in k\setminus\{0\}$, which is not a root of unity.}

Denote by  $S$ the set of roots of $\det A=0$. Then $S$ is a non-empty finite set of points in $\A^1(k)$.  Let $I$ be the set of fixed points of $g$. Then $I=\emptyset$ (if $g(x)=x+1$) or $I=\{0\}$ (if $g(x)=ax$; note that $a\in k^*$ is not a root of unity). Moreover $I$ is exactly the set of (pre)periodic points of $g.$

Denote by $\pi:\A^1_k\times\A^N_k\lra \A_k^1$ the projection onto the first coordinate. Then $\pi^{-1}(S)$ is the critical set of $f$. If $\pi(\alpha)\in I$ then Theorem~\ref{thmdmlslao} holds (similar to the case when $g(x)=x$ from Lemma~\ref{a root of unity automorphism}). 

{\bf Thus from now on, we assume that 
$\pi(\alpha)\not\in I$. So, there exists $M\geq 0$, such that for all $n\geq M$, we have that $\pi(f^n(\alpha))\not\in S.$}

Now, there exists a finitely generated $\Z$-algebra $R$  such that $f$, $C$, $\alpha$ and all points of $S$ are defined over $R$. Then $f^n(\alpha)$ are defined over $R$ for all $n\geq 0$. 
Suppose that $\OO_f(\alpha)\cap C$  is infinite. By Siegel's Theorem \cite[Theorems~8.2.4~and~8.5.1]{Siegel}, $C$ has at most two branches at the infinity.

Moreover, for all $i\geq 0$, denote by $C^{-i}$ the strict transform of $C$ under $f$; then we have
$$\{n\geq M\colon f^n(\alpha)\in C\}\subseteq \{n\colon M\le n\le M+i\}\cup \{n\geq M+i+1\colon f^{n-i}(\alpha)\in C^{-i}\}.$$  
Since $\OO_f(\alpha)\cap C$  is infinite, $\OO_f(\alpha)\cap C^{-i}$  is infinite for all $i\geq 0$. So, by Siegel's Theorem, $C^{-i}$ has at most two branches at the infinity for all $i\geq 0$.

\begin{lemma}
\label{S in I}
With the above notation, Theorem~\ref{thmdmlslao} holds if $S\subseteq I$.
\end{lemma}

\begin{proof}[Proof of Lemma~\ref{S in I}.]
If $S\subseteq I$, then $I=\{0\}$ and moreover, $f|_{\pi^{-1}(\A^1\setminus \{0\})}$ is unramified. Then we conclude our proof by using \cite[Theorem~1.3]{Bell2010}.
\end{proof}

{\bf Hence, from  now on, we assume that $S\not\subseteq I$.} 

So, there exists a point $s\in S\setminus I$ such that for all $n\geq 1,$ we have that $g^n(s)\not\in S$. Denote by $F_n:=f^n(\pi^{-1}(s))$ for each $n\geq 0$; then $F_n$ is a linear subspace of $$\pi^{-1}(g^n(s))\simeq \A^N$$ of dimension $d\leq N-1.$ If there exist three integers $1\leq n_1<n_2<n_3$ such that $C\cap \pi^{-1}(g^{n_i}(s)) \not\subseteq F_{n_i}$ (for $i=1,2,3$), then $C^{-n_3}$ has at least three branches at infinity, which is a contradiction. So there exists $B\geq 1$, such that for all integers $n\geq B,$ we have that 
$$C\cap \pi^{-1}(g^n(s))\subseteq F_n.$$

Let $\Gr_k(d,N)$ be the Grassmannian parametrizing all linear subvarieties of dimension $d$ contained in $\A^n_k$ and let $F:\A^1_k\times \Gr_k(d, N)\lra \A^1_k\times \Gr_k(d, N)$ be the birational map defined by $(x, V)\mapsto (g(x), A(x)(V))$, where for each $x\in k$, we denote by $A(x)(V)$ the image of $V$ under the linear map $A(x):\A^N_k\lra \A^N_k$.  We see that $F$ is well defined when $x\not\in S$. Let 
$$Z:= \{(x,V)\colon C\cap \pi^{-1}(x)\subseteq V\}.$$ 
Then $Z$ is a proper subvariety of $\A^1_k\times \Gr_k(d, N)$. For all $n\geq B$, we have $(g^n(s),F_n)\in Z.$ Let $W$ be the Zariski closure of $\{(g^n(s),F_n)\colon n\geq B\}$ in $\A^1_k\times \Gr_k(d, N)$. We have $\dim W\geq 1$ and $W\subseteq Z.$ At the expense of replacing $B$ by a larger integer, we may suppose that all irreducible components of $W$ have positive dimension. Then we have $F(W)=W.$ Moreover we have $\pi(W)=\A^1.$ For any $x\in \A^1$, denote by $W_x$ the fiber of $W$ over $x.$ Let 
$$U':=\left\{(x,y)\in \A^1_k\times\A^N_k\colon y\in \bigcap _{V\in W_x}V\right\}.$$ 
There exists a finite set $D$ of $\A^1(k)$, such that $U'\cap \pi^{-1}(\A^1\setminus D)$ is a vector bundle on $\A^1\setminus D$. Denote by $U$ the Zariski closure of $U'\cap \pi^{-1}(\A^1\setminus D).$ We have $C\subseteq U$ and $f(U)\subseteq U.$ Since $\OO_f(\alpha)\cap U$ is not empty, after replacing $\alpha$ by some $f^m(\alpha)$, we may suppose that $\alpha\in U.$

There exists an integer $B'\geq B$ such that for all $n\geq B'$, we have that  $g^n(s)\not\in D$. Then for all $n\geq B'$, 
$$U\cap \pi^{-1}(g^n(s))=U'\cap \pi^{-1}(g^n(s))\subseteq F_n.$$ 
So, $U$ is an irreducible proper subvariety of $\A^1\times \A_k^N$.

\begin{lem}\label{lemvecbunext}
Let $D$ be a finite set of $\A^1(k)$, let $U$ be an irreducible proper subvariety of $\A^1\times\A_k^N$. If $U\cap \pi^{-1}(\A^1\setminus D)$ is a vector bundle on $\A^1\setminus D$, then $U$ is a vector bundle over $\A^1.$
\end{lem}

\begin{proof}[Proof of Lemma \ref{lemvecbunext}.]
We prove this lemma by induction on $N.$
If $N=1$, then $U=\{y=0\}.$

Now we suppose that $N\geq 2.$ Denote by $\eta$ the generic point of $\A^1$; so, the generic fiber $U_{\eta}$ is a linear subspace of the generic fiber $\A^N_{k(x)}.$ There exists a hyperplane $H'\subseteq \A^N_{k(x)}$ containing $U_{\eta}$ and defined by $\sum_{i=1}^Na_i(x)y_i=0$. We may suppose that $a_i(x)\in k[x]$ for all $i=1,\dots,N$ and also that the polynomials $a_1(x),\dots,a_N(x)$ are coprime.  Denote by $H$ the subvariety of $\A^1\times\A_k^N$ defined by $\sum_{i=1}^Na_i(x)y_i=0$. Then $H$ is a vector bundle on $\A^1$ and $U\subseteq H$. Quillen-Suslin's theorem (see \cite{Quillen}) yields that all vector bundles on $\A^1_k$ are trivial. So $H$ is isomorphic to $\A^1_k\times \A^{N-1}_k$, which allows us to conclude our proof by the induction hypothesis.
\end{proof}

Thus Lemma~\ref{lemvecbunext} yields that $U$ is a vector bundle on $\A^1_k$. Applying again Quillen-Suslin's theorem we obtain that $U$ is trivial on $\A_k^1$. 
By using the induction hypothesis on the  restriction of $f$ on $U$, we conclude our proof of Theorem~\ref{thmdmlslao}.


\section{Proof of Theorem~\ref{thmdmlsenonlinear}}
\label{section proofs 2}


We work under the hypothesis of Theorem~\ref{thmdmlsenonlinear}. Denote by $\pi:\P^1_k\times \A^N_k\to \P^1_k$ the projection to the first factor.

\begin{lemma}
\label{preperiodic non-automorphism}
Theorem~\ref{thmdmlsenonlinear} holds if $\pi(\alpha)$ is preperiodic. 
\end{lemma}

\begin{proof}[Proof of Lemma~\ref{preperiodic non-automorphism}.]
After replacing $\alpha$ by $f^m(\alpha)$ for a suitable $m\geq 0$ and $f$ by some iterate of $f$, we may assume that $\pi(\alpha)$ is fixed by $g$.
Since $f|_{\pi(\alpha)}$ is a linear morphism, then \cite[Theorem~4.1]{Bell2010} provides the desired conclusion (see also our proof of Lemma~\ref{a root of unity automorphism}). 
\end{proof}

{\bf So, from now on, we assume that $\pi(\alpha)$ is not preperiodic.}

We split our proof into two cases (which will be proved in Subsections~\ref{subsec:1}~respectively~\ref{subsec:2}) depending on whether $f$ is dominant, or not.

\subsection{The case where the map is dominant}
\label{subsec:1}

We first treat the case when $f$ is dominant; we will see in Subsection~\ref{subsec:2} that the general case may be reduced to this special case.

Denote by $S$ the set of points $x\in \P^1(k)$ that $f$ is not locally \'etale along $\pi^{-1}(x).$
In particular, $\pi(I(f))\subseteq S$ and $S$ is finite. 
 
There exists a subfield $K_1$ of $k$ which is finitely generated over $\bar{\Q}$ such that $g,f,S$ and $\alpha$ are all defined over $K_1$. 
Now we may assume that $k=\overline{K_1}.$

There exists a finitely generated $\bar{\Q}$-subalgebra $A$ of $K_1$ such that $K_1=\Frac A$. Let $$A:=\bar{\Q}[x_1,\dots,x_{m_0}]/(F_1,\dots, F_{\ell_0})$$
for some $m_0,\ell_0\in\N$ and some suitable polynomials $F_i$; we may assume that $\Spec A$ is smooth. There exists a finite extension $M/\Q$ such that all coefficients of the polynomials $F_1,\dots,F_{\ell_0}$ are contained in $M$. Furthermore, we let $\cO_M$ be a finitely generated subring of $M$, whose fraction field equals $M$, such that each coefficient of $F_i$ is contained in $\cO_M$. Let  $$R:=\cO_{M}[x_1,\dots,x_{m_0}]/(F_1,\dots, F_{\ell_0});$$ 
then $R$ is a subring of $A$, which is finitely generated over $\Z$. Also, we  let $K:=\Frac(R)$. At the expense of replacing $M$ by another finitely generated extension, we may assume that $g,f,S$ and $\alpha$ are all defined over $K.$
We may identify $f$ as $f_K\otimes_Kk$ where $f_K$ is a rational self-map on $\P^1_{K}\times \A^N_{K}.$

The next proposition is the key technical ingredient for our proof. Proposition~\ref{lemexiredu} yields the existence of a nonarchimedean place $v$ which meets several good hypotheses regarding the reduction modulo $v$ for our dynamical system; in particular, this allows us to find a suitable $p$-adic analytic parametrization of our orbit (using the $p$-adic arc lemma, as constructed in \cite[Chapter~4]{Bell2016}).
 
\begin{pro}\label{lemexiredu}
There exists a finite extension $L$ over $\Q$, a nonarchimedean place $v$ of $L$, and an embedding $K\hookrightarrow L_v$ where $L_v$ is the $v$-adic completion of $L$ (with ring of $v$-adic integers denoted by $\cO_v$),  
a rational self-map $f_{\cO_v}:\P^1_{\cO_v}\times \A^N_{\cO_v}\dashrightarrow\P^1_{\cO_v}\times \A^N_{\cO_v}$ of the form $(x,y)\mapsto (g_{\cO_v}, A_{\cO_v}(x)(y))$ such that 
\begin{points}
\item $f_{L_v}=f_K\otimes_KL_v$ where $f_{L_v}$ is the restriction of $f_{\cO_v}$ on the generic fiber;
\item the restriction $f_v$ of $f_{\cO_v}$ on the special fiber is a dominant rational self-map on $\P^1_{\F_v}\times \A^N_{\F_v}$ where $\F_v$ is the residue field of $L_v;$
\item the restriction $g_v$ of $g_{\cO_v}$ on the special fiber is an endomorphism of $\P^1_{\F_v}$ of degree $\deg g$;
\item the set $S_v$ of points $x\in \P^1(\F_v)$ with the property that $f_{\cO_v}$ is not locally \'etale along $\pi_p^{-1}(x)$ is the specialization of $S$ on the special fiber;
\item there exists $r\geq 0$, such that for $n\geq r$, the specialization of $g_K^n(\pi(\alpha))$ on the special fiber is not contained in $S_v.$
\end{points}
\end{pro}

\begin{proof}[Proof of Proposition~\ref{lemexiredu}.]
 Observe that $f_K$ induces a rational self-map $f_{R}:\P^1_{R}\times \A^N_{R}\dashrightarrow\P^1_{R}\times \A^N_{R}$ of the form $(x,y)\mapsto (g_{R}, A_{R}(x)(y))$ such that 
the restriction of $f_{R}$ on the generic fiber is $f_K.$ 

By shrinking $\Spec R$, we may assume that $R$ is regular and moreover, assume the following properties hold:  
\begin{itemize}
\item at every point $t\in \Spec R$, the restriction $f_t$ of $f_{R}$ on the special fiber at $t$ is a dominant rational self-map on $\P^1_{\kappa(t)}\times \A^N_{\kappa(t)}$ where $\kappa(t)$ is the residue field at $t$;
\item the restriction $g_t$ of $g_{R}$ on the special fiber is an endomorphism of $\P^1_{\kappa(t)}$ of degree $\deg g$; and
\item the set $S_t$ of points $x\in \P^1(\kappa(t))$ such that $f_R$ is not locally \'etale along $\pi_t^{-1}(x)$ is the specialization of $S$ on the special fiber at $t$.
\end{itemize}
In particular, we have that  $g_R:\P^1_R\to \P^1_R$ is a regular endomorphism.
 
Let $q:=\pi(\alpha)$; note that $q$ is not preperiodic under $f_K$, according to our earlier assumption. 
 Let $q_R$ be the Zariski closure of $q$ in $\P^1_R.$ 
 For any point $t\in \Spec R$, denote by $\P^1_t$ the fiber of $\P^1_R$ at $t$ and also, let $q_t$ be the unique point in $q_R\cap \P^1_t.$ Indeed, $t\mapsto q_t$ is a section from $\Spec R$ to $\P^1_R.$
 
The next Lemma is employed in our proof of Proposition~\ref{lemexiredu}. 
\begin{lem}\label{lemexistnonrepredu}
Let $A$ be a finitely generated integral $\bar{\Q}$-algebra with $\Frac(A)=B$. Let $g_A:\P^1_A\to \P^1_A$ be a regular and dominant endomorphism of degree  greater than $1$.
 Let $q\in \P^1(B)$ be a point on the generic fiber which is not preperiodic under $g_B.$ Then there exists a $\bar{\Q}$-point $c\in \Spec A$ such that the specialization $q_c$ of $q$ on the fiber at $c$ is not preperiodic.
\end{lem}

\begin{proof}[Proof of Lemma~\ref{lemexistnonrepredu}.]
We let a tower of field extensions $\bar{\Q}=:B_0\subset B_1\subset \cdots \subset B_m:=B$ such that  for each $i=0,\dots, m-1$, we have that $B_{i+1}/B_i$ has transcendence degree equal to $1$. Without loss of generality, we may even assume that each $B_{i+1}$ is the function field of some smooth, projective curve defined over $B_i$.  Therefore, it suffices to prove the following statement.
\begin{claim}
\label{claim:nonpreperiodic}
{\em Let $L_0$ be a field of characteristic $0$ and let $L_1$ be the function field $L_0(C)$ for some smooth, projective curve $C$ defined over $L_0$. Let $g$ be a rational function defined over $L_1$ of degree larger than $1$ and let $x\in \P^1(L_1)$ be a non-preperiodic point under the action of $g$. Then there exists a place $\fp$ of  of the function field $L_1/L_0$ such that $g$ has good reduction at the place $\fp$ and moreover, the specialization of $x$ at $\fp$ (denoted $x_\fp$) is not preperiodic under the action of the corresponding specialization $g_\fp$ of $g$.}
\end{claim}

\begin{proof}[Proof of Claim~\ref{claim:nonpreperiodic}.]
The proof is similar to the one employed for \cite[Proposition~6.2]{GTZ-Inventiones}; the only change is replacing the reference \cite{Benedetto} by \cite{Baker}. Indeed, generalizing the results of \cite{Benedetto}, Baker \cite{Baker}  showed that the canonical height $\hat{h}_g(x)$ is strictly positive if $g$ is not isotrivial for the function field $L_1/L_0$ (note that $x$ is not preperiodic). For more details regarding isotrivial rational functions, see \cite[Section~6]{GTZ-Inventiones}, while for more details about the canonical height $\hat{h}_g(\cdot)$ associated to a rational function of degree larger than one, see \cite{Call-Silverman}. We note that Chatzidakis-Hrushovski \cite{C-Z-1, C-Z-2} proved a further generalization of Baker's result \cite{Baker} to polarized algebraic dynamical systems.

Now, there are two cases: either $\hat{h}_g(x)=0$, or not. We recall (see \cite{Call-Silverman}) that $\hat{h}_g(\cdot )$ is the canonical height (on the generic fiber) constructed with respect to the places of the function field $L_1/L_0$.

{\bf Case 1.} $\hat{h}_g(x)=0$.

In this case, because $x$ is not preperiodic, \cite{Baker} yields that $g$ is isotrivial and moreover, if $\mu:\P^1_{\bar{L_1}}\lra \P^1_{\bar{L_1}}$ is an automorphism such that $\mu^{-1}\circ g\circ \mu$ is defined over $\bar{L_0}$, then also $\mu(x)\in\P^1_{\bar{L_0}}$. Therefore,  all but finitely many places $\fp$ of the function field $L_1/L_0$ satisfy the conclusion of our hypothesis.

{\bf Case 2.} $\hat{h}_g(x)>0$.

In this case, let $h_C(\cdot )$ be a Weil height for the $\overline{L_0}$-points on $C$ associated to an ample divisor (of degree $1$). Then using \cite{Call-Silverman} (see also the proof of \cite[Proposition~6.2]{GTZ-Inventiones}), we obtain that each specialization at a place corresponding to a point $\fp$ of $C$ of sufficiently large height (i.e., $h_C(\fp)\gg 0$) guarantees the concluson of Claim~\ref{claim:nonpreperiodic}. Indeed, Call-Silverman \cite{Call-Silverman} proved that 
\begin{equation}
\label{eq:nonisotrivial}
\lim_{h_C(\fp)\to\infty} \frac{\hat{h}_{g_{\fp}}(x_{\fp})}{h_C(\fp)}=\hat{h}_g(x),
\end{equation}
where $\hat{h}_{g_{\fp}}(\cdot)$ is the canonical height associated to the specialization $g_\fp$ of $g$ at the place $\fp$ of good reduction for $g$ (here we construct $\hat{h}_{g_{\fp}}(\cdot )$ with respect to the Weil height on $\bar{L_0}$, since $L_0$ is either $\bar{\Q}$ or a function field over $\bar{\Q}$). Our assumption that $\hat{h}_g(x)>0$ coupled with equation \eqref{eq:nonisotrivial} yields that $\hat{h}_{g_\fp}(x_\fp)>0$ if $h_C(\fp)\gg 0$, and therefore $x_\fp$ is not preperiodic under the action of $g_{\fp}$ (since each preperiodic point would have canonical height equal to $0$; see \cite{Call-Silverman}).
\end{proof}

Applying Claim~\ref{claim:nonpreperiodic} repeatedly for each extension $B_{i+1}/B_i$ for $i<m$ yields the desired conclusion of Lemma~\ref{lemexistnonrepredu}.
\end{proof}
 
We return to the proof of Proposition~\ref{lemexiredu}. 
Lemma~\ref{lemexistnonrepredu} yields the existence of a point $c\in \Spec A$ such that $q_c$ is not preperiodic.
There exists a finite extension $L$ over $M$ such that $c$ is defined over $L$. 
By replacing $R$ by $R\otimes_{\cO_M}\cO_L$, we may assume that there exists a point $c\in \Spec R$ such that $\kappa(c)=L$ and $q_c\in \P^1_c(L)$ is not preperiodic.

By replacing $L$ by a suitable finite extension of $L$, we may assume that there exists a nonexceptional periodic point $\beta_c\in \P^1_c(L)$ whose orbit $\OO_{f_c}(\beta_c)$ does not meet 
$S_c.$

Observe that $c$ is a point on the generic fiber of $\Spec R\to \Spec \cO_L.$ The Zariski closure $Z_c$ of $c$ in $\Spec R$ is isomorphic to a Zariski open set $U$ of $\Spec \cO_L.$
Denote by $\delta: \P^1_R\to \Spec R$ the structure morphism.  Then $\delta^{-1}(Z_c)\to Z_c$ is isomorphic to the scheme $\P^1_{U}\to U.$ 
For any point $t\in Z_c$, denote by $\beta_t$ the specialization of $\beta_c$ in $\P^1_t$. Then $\beta_t$ is periodic and $\OO_{f_t}(\beta_t)$ is the specialization of $\OO_{f_c}(\beta_c)$.
By \cite[Lemma 4.1]{BGKT-MA}, there exists a point $e\in Z_c$  such that $\OO_{f_e}(\beta_e)\cap S_e=\emptyset$ and there exists $s>0$ such that $f_e^s(q_e)=\beta_e.$  In particular for all $n\geq s$, we have that  $f_e^n(q_e)\notin S_e.$

The image of $e$ in $U$ is a prime of $\cO_L$ which defines a nonarchimedean place $v$ of $L$.
Let $A_{\cO_{v}}:=R\otimes_{\cO_L}\cO_{v}.$ We may view $e$ as a point on the special fiber of $\Spec A_{\cO_{v}}\to \Spec \cO_v$ and $c$ as a point on the generic fiber of $\Spec A_{\cO_{v}}\to \Spec \cO_v$.   Let $d:=\dim A$ and denote by $A_{L_v}$ the generic fiber of $\Spec A_{\cO_{v}}\to \Spec \cO_v$. The set $V_e$ of points in $\Spec A_{L_v}(L_v)$ whose specialization is $e$ is a $v$-adic neighbourhood of $c$ and it is isomorphic to the polydisc $\cO_v^d.$
Denote by $A_L:=R\otimes_{\cO_L}L.$ We have an inclusion $\tau:A_L\hookrightarrow A_{L_v}.$ For any nonzero element $F\in A_L$, denote by $Z(F)$ the zeros of $F$ in $\Spec A_{L_v}(L_v).$ 

Since $A_L$ is countable, by Baire category theorem, $\cup_{F\in A_L\setminus \{0\}}Z(F)$ is nowhere dense in $A_{L_v}(L_v)$ with respect to the $v$-adic topology. Then there exists a point 
$$z\in V_e\setminus (\cup_{F\in A_L\setminus \{0\}}Z(F)).$$ 
The point $z$ defines a morphism $\chi_z:A_{L_v}\to L_v.$ 
The Zariski closure $Z_z$ of $z$ in $A_{\cO_v}$ is isomorphic to $\Spec \cO_v.$ Denote by 
$$\sigma:\P^1_{A_{\cO_v}}\times \A^N_{A_{\cO_v}}\to \Spec A_{\cO_v}$$ 
the structure morphism. Then $\sigma^{-1}(Z_z)\to Z_z$ is a scheme isomorphic to $\P^1_{\cO_v}\times \A^N_{\cO_v}\to \Spec \cO_v.$
The restriction of $f_R\otimes_{\cO_L} \cO_v$ on $\sigma^{-1}(Z_z)$ gives a rational self-map 
$$f_{\cO_v}:\P^1_{\cO_v}\times \A^N_{\cO_v}\dashrightarrow\P^1_{\cO_v}\times \A^N_{\cO_v},$$  
of the form $(x,y)\mapsto (g_{\cO_v}, A_{\cO_v}(x)(y))$.

We have a morphism $\iota:=\chi_z\circ\tau: A_{L}\to L_v.$ Since $z\not\in \cup_{F\in A_L\setminus \{0\}}Z(F)$,  we have $\ker\iota=0$ and then $\iota$ is injective. Then $\iota$ extends to an embedding from $K=\Frac R=\Frac A_L$ to $L_v.$ It follows that $f_{L_v}=f_K\otimes_KL_v$; this yields (i) from the conclusion of Proposition~\ref{lemexiredu}. The properties~(ii),(iii),(iv) hold  because such properties hold for all points in $\Spec R$. The property~(v) holds because the specialization of $z$ on the speical fiber is $e$ and for all $n\geq N$, we have that $f_e^n(q_e)\not\in S_e.$

This concludes our proof of Proposition~\ref{lemexiredu}.
\end{proof}

Now, we are ready to finish the proof of Theorem~\ref{thmdmlsenonlinear}.

We apply Proposition~\ref{lemexiredu} and therefore, by replacing $\alpha$ with  $f^r(\alpha)$ (as in property~(v) from the conclusion of Proposition~\ref{lemexiredu}), we may assume that the specialization of $\pi(f^n(\alpha))$ on the special fiber $\P^1_{\F_v}$ is not contained in $S_v$ (for each $n$). 

For each point $q\in \P^1(L_v)\times \A^N(L_v)$(respectively $q\in \P^1(L_v)$),  denote by $\bar{q}$ its specialization in $\P^1(\F_v)\times \A^N(\F_v)$ (respectively $\P^1(\F_v)$). At the expense of replacing $\alpha$ by $f^r(\alpha)$, we may suppose that $\overline{\pi(f^n(\alpha))}\notin S_v$ for all $n\geq 0$. It follows that $f_v^n(\bar{\alpha})=\overline{f^n(\alpha)}$ for all $n\geq 0.$ 

Because $\P^1(\F_v)$ is finite, at the expense of replacing $\alpha$ by $f^m(\alpha)$ for some $m\geq 0$ and also replacing $f$ by a suitable iterate, we may assume that $\bar{\alpha}$ is fixed by $f_v$. 
Then the set $V$ of points $q\in \P^1(L_v)\times \A^N(L_v)$ satisfying $\bar{q}=\bar{\alpha}$ is a $v$-adic open neighbourhood of $\alpha$. Moreover, $V$ is isomorphic to $\cO_v^{N+1}$ and it is also invariant by $f$. Then $f|_V$ induces an analytic self-map on $\cO_v^{N+1}.$ Since $f$ is \'etale along the fiber at $\overline{\pi(\alpha)}$, we have that $df_v(\alpha)$ is invertible. Then $f|_V$ is an isomorphism of $V$.  

Let $p>0$ be the characteristic of $F_v$. Without loss of generality, we may identify $V$ with $\cO_v^{N+1}$. At the expense of replacing $f$ by a suitable iterate, we may further assume that $f|_v= \id \mod p^2.$ 
By \cite[Theorem 1]{Poonen2014} (which uses the ideas introduced in \cite{Bell2010}), there exists a $v$-adic analytic map 
$$h:\cO_v\to V$$ such that for any $n\in \N$, we have $h(n)=f^n(\alpha).$ Then $h^{-1}(Y)$ is an analytic subset of $\cO_v$. Since $\cO_v$ is compact, $h^{-1}(Y)$ is either finite, or it is all of $\cO_v.$
If there exists an infinite sequence $n_1<n_2<\dots$ such that $f^{n_i}(\alpha)\in Y$,  then $h^{-1}(Y)$ is infinite and so, it must equal  $\cO_v$. Hence $f^n(\alpha)\in Y$ for all $n\geq 0.$ Since $Y$ is irreducible and $\OO_f(\alpha)$ is Zariski dense in $Y$, we obtain that $Y$ is periodic under $f$, which concludes our proof of Theorem~\ref{thmdmlsenonlinear}.

\subsection{The general case}
\label{subsec:2}

Since in Subsection~\ref{subsec:1} we proved Theorem~\ref{thmdmlsenonlinear} if $f$ is dominant, we are left now with the case $f:\P^1\times \A^N\dra\P^1\times \A^N$ is not dominant. 

For any $n\geq 0$, we let $\mathcal{I}_n:=f^n(\P^1_k\times \A^N_k)$. Then $\mathcal{I}_n\supseteq \mathcal{I}_{n+1}$ for each $n\ge 0$, and moreover, each $\mathcal{I}_n$ is a subbundle of the trivial vector bundle 
$$\pi:\P^1_k\times \A^N_k\to \P^1_k.$$ 
Let $Z:=\cap_{n\geq 0}\mathcal{I}_n$; then $Z$ is a subbundle of $\pi$ and furthermore, it is invariant under $f$. So, there exists $\ell\geq 0$ such that $$Z:=\cap_{n= 0}^\ell f^n(\P^1_k\times \A^N_k).$$ 
Also, we have that $f|_Z:Z\dashrightarrow Z$ is dominant. 
After replacing $\alpha$ by $f^\ell(\alpha)$, we may assume that $\alpha\in Z$. Since $\dim Y\geq 1$ and $\OO_f(\alpha)\cap Y$ is Zariski dense in $Y$, we have that  $Y\subseteq Z$. 

We note that $\pi|_Z:Z\to \P^1_k$ may be a nontrivial vector bundle on $\P^1_k$. So we may not be able to apply directly the results of Subsection~\ref{subsec:1} in order to conclude our proof. However, there exists a trivial line bundle $\pi':Z'\to \P^1_k$, a skew-linear rational self-map $f':Z'\dashrightarrow Z'$ and a birational map $\phi:Z\dashrightarrow Z'$ satisfying the following propertis:
\begin{itemize}
\item $\phi\circ f = f'\circ \phi$;
\item $\pi'\circ f'=g\circ \pi'$; 
\item $\pi|_Z=\pi'\circ\phi$; and 
\item on the generic fiber, $\phi$ is a linear isomorphism.  
\end{itemize}
Note that $I(\phi)$ is contained in a union $U_0$ of  finite fibers of $\pi|_Z$.

Since $\pi(\alpha)$ is not preperiodic, after replacing $\alpha$ by $f^m(\alpha)$ for a suitable $m\geq 0$, we may assume that $f^n(\alpha)\not\in U_0$ for all $n\geq 0.$ It follows that $Y\not\subseteq U_0$ and $f^{'n}(\phi(\alpha))\notin I(f')$ for all $n\geq 0.$ Then $\OO_{f'}(\phi(\alpha))\cap \phi(Y)$ is Zarski dense in $\phi(Y)$. Since $f'$ is dominant, the result proven in  Subsection~\ref{subsec:1} implies that $\phi(Y)$ is periodic under $f'$, which in turn, implies that $Y$ is periodic under $f$ (because $\phi$ is a birational map).  This concludes our proof of Theorem~\ref{thmdmlsenonlinear}.

 \begin{rem}
\label{rem:no extension}
We note that the proof of Theorem~\ref{thmdmlsenonlinear} does not work when $g$ is an automorprhism; we describe next the key point in the proof of Theorem~\ref{thmdmlsenonlinear}. Assume that $f$ and $\alpha$ are defined over $\Q$, $f:\P^1\times \A^N\lra \P^1\times \A^N$ is a skew-linear dominant, self-map such that $f^n(\alpha)\notin I(f)$ for all $n\geq 0$, and furthermore, assume $\pi(\alpha)$ is not preperiodic.  Then there exist a prime $p$, a $p$-adic polydisc $V\subset  \P^1(\Q_p)\times \A^N(\Q_p)$ and  integers $m,\ell\geq 1$ such that:
 \begin{points}
 \item $f^m(\alpha)\in V$; and
 \item $f^\ell$ is well defined on $V$ and moreover, $f^\ell(V)=V.$ 
 \end{points}

It is easy to show that if $f:\A^1\times \A^1\lra \A^1\times \A^1$ is the morphism defined by $(x,y)\mapsto (x+1, xy)$ and the starting point is  $\alpha=(0,0)$, then such a polydisc $V$ does not exist.
 \end{rem}


\section{Picard-Vessiot extensions inside the ring of sequences\\ Appendix by Michael Wibmer} \label{sec:appendix}

In this appendix we use Theorem \ref{thmdmlsenonlinear} (more specifically Corollary \ref{corsmlnonline}) to show that for certain linear difference equations there exists a Picard-Vessiot extension inside the ring of sequences. 

As in the previous sections we work over an algebraically closed field $k$ of characteristic zero. A \emph{difference ring (or field)} is a commutative ring (or field) $R$ together with a distinguished ring endomorphism $\s\colon R\to R$. A morphism of difference rings is a morphism of rings that commutes with the distinguished endomorphism. The \emph{constants} of a difference ring $R$ are $R^\s=\{a\in R|\ \s(a)=a\}$. For more background on difference algebra and linear difference equations the reader is referred to \cite{Levin:difference} and \cite{Put-Singer}.

As in \cite[Example 1.3]{Put-Singer} we consider the ring $\seq_k$ of sequences in $k$: two sequences $(a_n)_{n\in\N_0}$ and $(b_n)_{n\in\N_0}$ are identified if $a_n=b_n$ for $n\gg 0$. Addition and multiplication is componentwise and $\seq_k$ is considered as a difference ring by shifting a sequence one step to the left, i.e., $\s((a_n)_{n\in\N_0})=(a_{n+1})_{n\in\N_0}$.

Let $k(x)$ be the field of rational functions over $k$. We fix a non-constant endomorphism $g\colon \P^1_k\to\P^1_k$, i.e., $g\in k(x)\setminus k$ and an $\alpha\in\P^1_k(k)$ such that $\OO_g(\alpha)$ is infinite. Then $h(g^n(\alpha))\in k$ for $h\in k(x)$ and $n\gg0$. Thus we have a well-defined map
\begin{equation} \label{eq: map}
k(x)\to\seq_k,\  h\mapsto (h(g^n(\alpha)))_{n\in\N_0}.
\end{equation}
We consider $k(x)$ as a difference field via $\s(h(x))=h(g(x))$, so that (\ref{eq: map}) becomes a morphism of difference rings. In the sequel we will always consider $\seq_k$ as a difference ring extension of $k(x)$ via this embedding.

We are interested in linear difference equations $\s^\ell(y)+h_{\ell-1}\s^{l-1}(y)+\ldots+h_0y=0$ with $h_0,\ldots,h_{\ell-1}\in k(x)$, or more generally in first order system $\s(y)=Ay$ with $A\in M_{\ell,\ell}(k(x))$. The Galois theory of linear difference equations (also known as Picad-Vessiot theory) associates a linear algebraic group to such an equation. This Galois group measures the algebraic relations among the solutions of the linear difference equation.
The Galois group can be defined as the automorphism group of a Picard-Vessiot extension for the linear difference equation. The standard reference for the Galois theory of linear difference equations is \cite{Put-Singer}. There it is always assumed that $\s$ is an automorphism. In our situation, this assumption is only satisfied if $\deg(g)=1$. However, the assumption that $\s$ is an automorphism is not necessary for developing the Galois theory of linear difference equations. (See e.g., \cite{Wibmer:OnTheGaloisTheoryOfStronglyNormal} and \cite{Maier:Nori}).

It is a characteristic of the Galois theory of difference equations that Picard-Vessiot extensions may contain zero divisors. Instead of difference field extensions of $k(x)$ one considers certain difference rings that have field-like properties: a difference ring $L$ is a \emph{difference pseudo field} if
\begin{enumerate}
	\item $L$ is $\s$-simple, i.e., if $I$ is an ideal of $L$ with $\s(L)\subseteq L$, then $L=\{0\}$ or $L=R$,
	\item every non-zero-divisor of $L$ is a unit of $L$ and
	\item $L$ is Noetherian.
\end{enumerate}

One can show (\cite[Prop. 1.3.2]{Wibmer:thesis}) that every difference pseudo field $L$ is of the form $L=e_1L\oplus\ldots\oplus e_rL$, where $e_1,\ldots,e_r$ are orthogonal idempotent elements of $L$ such that
\begin{enumerate}
	\item $\s(e_1)=e_2,\s(e_2)=e_3,\ldots,\s(e_n)=e_1$ and
	\item the $e_i L$'s are fields.	
\end{enumerate}

Conversely, any difference ring of the form  $L=e_1L\oplus\ldots\oplus e_rL$ with (1) and (2) satisfied, is a difference pseudo field.

If $R$ is a $k(x)$-algebra and $C\subseteq R$ we define $$k(x)(C)=\{ab^{-1}|\ a,b\in k(x)[C],\ b \text{ is a unit in } R\}\subseteq R.$$ In particular, if $R$ is a field extension of $k(x)$, then $k(x)(C)$ is the field extension of $k(x)$ generated by $C$.

\begin{defi}
	Let $A\in\Gl_\ell(k(x))$. A difference pseudo field extension $L/k(x)$ is a \emph{Picard-Vessiot extension} for $\s(y)=Ay$ if
	\begin{enumerate}
		\item $L$ is generated by the entries of a fundamental solution matrix, i.e., there exists $Y\in\Gl_\ell(L)$ with $\s(Y)=AY$ such that $L=k(x)(Y_{ij})$ and  
		\item $L^\s=k(x)^\sigma(=k)$.
	\end{enumerate}	
\end{defi}

Since for any $A\in\Gl_\ell(k(x))$ the linear difference equation $\s(y)=Ay$ has a fundamental solution matrix in $\seq_k$ (cf. proof of Theorem \ref{theo:existsPV} below) and moreover $(\seq_k)^\sigma=k$, it is natural to ask if there exists a Picard-Vessiot extension for $\s(y)=Ay$ inside $\seq_k$. In \cite{Wib-JEMS} this question was considered for the special case $g(x)=x+1$ and $\alpha=0$. While the problem remains open in this important special case, the answer is affirmative if $\deg(g)\geq 2$.


%
%
%

%

\begin{thm} \label{theo:existsPV} If $\deg(g)\geq 2$, then for any $A\in\Gl_\ell(k(x))$ there exists a Picard-Vessiot extension for $\s(y)=Ay$ in $\seq_k$.
\end{thm}
\begin{proof}
	Let us first show that there exists a fundamental solution matrix $Y\in\Gl_\ell(\seq_k)$. We may choose $n_0\in\nn$ such that no entry of $A$ has a pole in $\{g^{n_0}(\alpha), g^{n_0+1}(\alpha),\ldots\}$ and $\det(A)$ is non-zero on $\{g^{n_0}(a), g^{n_0+1}(a),\ldots\}$. We set $Y(n_0)=I_\ell$ and for $n\geq n_0$ we define recursively $Y(n+1)=A(g^n(\alpha))Y(n)$. 
	Since $\det(Y(n))\neq 0$ for $n\geq n_0$ we see that $Y\in\Gl_\ell(\seq_k)$. Moreover, $\s(Y)=AY$ by construction.
	
	Since $(\seq_k)^\s=k$, it suffices to show that $L=k(x)(Y_{ij})\subseteq \seq _k$ is a 
	difference pseudo field. For this, the crucial step is to show that a non-zero-divisor $a$ of $R=k(x)[Y_{ij},1/\det(Y)]$ is a unit in $\seq_k$.

	A sequence $a\in\seq_k$ satisfies a linear difference equation over $k(x)$ if and only if it is contained in a finite dimensional  $\sigma$-stable $k(x)$-subspace of $\seq_k$. It follows that sums and products of sequences that satisfy a linear difference equation also satisfy a linear difference equation. The $k(x)$-subspace of $\seq_k$ generated by $Y_{ij},\ 1\leq i,j\leq \ell$ and $1/\det(Y)$ is stable under $\s$. Therefore, every element $a\in R=k(x)[Y_{ij},1/\det(Y)]$ satisfies a linear difference equation over $k(x)$.
	
	Suppose $a$ is a non-zero-divisor of $R$ but not a unit in $\seq_k$. Then $\{n\in\nn_0|\ a_n=0\}$ is an infinite set, since otherwise $a$ would be a unit in $\seq_k$. According to Corollary~ \ref{corsmlnonline}, the set $\{n\in\nn_0|\ a_n=0\}$ contains an arithmetic progression of the form $c+d\nn_0$ with $c,d\in\nn_0$ and $d\geq 1$. Therefore $a\s(a)\ldots\s^d(a)=0$. Since $R$ is stable under $\s$, this contradicts our assumption that $a$ is a non-zero-divisor of $R$.
Thus we have shown that every non-zero-divisor of $R=k(x)[Y_{ij},1/\det(Y)]\subseteq\seq_k$ is a unit in $\seq_k$. 

It follows that $L=k(x)(Y_{ij},1/\det(Y))=k(x)(Y_{ij})$ agrees with the total ring of fractions of $R$, i.e., $L$ is the localization of $R$ by all non-zero-divisors of $R$. Since $R$ is reduced, this implies that $L$ is a finite direct sum of fields. Say $L=e_1L\oplus\ldots\oplus e_rL$ with $e_1,\ldots,e_r$ orthogonal idempotent elements of $L$ and the $e_iL$'s are fields.

Note that $L$ is stable under $\s\colon \seq_k\to\seq_k$ because $\s$ maps units to units. Moreover, $\s$ is an automorphism of $\seq_k$, so in particular it is injective on $L$ and therefore induces a permutation on the primitive idempotent elements $e_1,\ldots,e_r$ of $L$. If $(i_1\ldots i_s)$ is a cycle occurring in the cycle decomposition of this permutation, then $e_{i_1}+\ldots+e_{i_s}$ is a constant idempotent element. However, the only constant idempotent elements of $\seq_k$ are $1$ and $0$. This shows that the permutation induced by $\s$ on $e_1,\ldots,e_r$ is an $r$-cycle. So after renumbering the $e_i$'s if necessary we have $\s(e_1)=e_2,\ldots,\s(e_r)=e_1$. Thus $L$ is a difference pseudo field.
\end{proof}




\end{document}